\documentclass[11pt]{article}

\usepackage{amsmath,amsthm}
\usepackage{mathtools}

\usepackage{ifxetex,ifluatex}
\ifxetex
\usepackage{xltxtra}
\else
\ifluatex
\usepackage{luatextra}
\else
\usepackage{amssymb}
\usepackage[utf8]{inputenc}
\usepackage[T1]{fontenc}
\usepackage{mathpazo,tgcursor,tgpagella}
\usepackage{textcomp}
\csname else\expandafter\endcsname
\romannumeral -`0
\fi
\fi
\iftrue\relax%
\defaultfontfeatures{Ligatures=TeX}
\usepackage[math-style=TeX, vargreek-shape=TeX]{unicode-math}
\AtBeginDocument{%
  \let\setminus\smallsetminus
}
\fi

\usepackage[english=american]{csquotes}
\usepackage[a4paper,margin=1in]{geometry}
\usepackage{xspace}
\usepackage{authblk}
\usepackage[authoryear]{natbib}
\usepackage[compact]{titlesec} 
\usepackage{enumitem}

\usepackage[bookmarksnumbered,bookmarksopen,
unicode]{hyperref}

\hypersetup{
  pdftitle={A short proof for the polyhedrality of the
    Chvátal–Gomory closure of a compact convex set},
  pdfauthor={Gábor Braun,
    Sebastian Pokutta},
  pdfkeywords={Chvátal–Gomory closure, polyhedrality, compactness,
    convex bodies},
  pdfsubject={MSC2010 Primary 90C25; Secondary 90C05, 90C10}}

\newcommand{\face}[1]{\left\{#1\right\}}
\newcommand{\gc}{CG\xspace}
\newcommand{\ffloor}[1]{\left\lfloor#1\right\rfloor}
\newcommand{\Z}{\mathbb{Z}}
\newcommand{\Q}{\mathbb{Q}}
\newcommand{\R}{\mathbb{R}}
\newcommand{\N}{\mathbb{N}}
\newcommand{\tnorm}[1]{\left\lVert#1\right\rVert}

\newcommand*{\aff}[1]{\operatorname{aff}(#1)}

\newcommand*{\unitvector}[1]{\widehat{#1}}

\newcommand*{\abs}[1]{\left|#1\right|}

\newtheorem{thm}{Theorem}
\newtheorem{lem}[thm]{Lemma}

\newtheorem{cor}[thm]{Corollary}

\theoremstyle{remark}
\newtheorem{rem}{Remark}

\title{A short proof for the polyhedrality of the \\
  Chvátal–Gomory closure of a compact convex set%
  }

\date{}
\author{Gábor Braun}
\author{Sebastian Pokutta}

\affil{ISyE,
  Georgia Institute of Technology,
  Atlanta, GA 30332, USA.
  \textit{Email:}~\{gabor.braun,~sebastian.pokutta\}@isye.gatech.edu}

\begin{document}

\maketitle

\begin{abstract}
Recently Schrijver’s open problem, whether the Chvátal-Gomory closure of an
irrational polytope is polyhedral was answered independently in the
seminal works of \cite{DDV2011} and \cite{ds2010}; the former even
applies to general compact convex sets. We present a very short, easily
accessible proof.
\end{abstract}

\section{Introduction}
\label{sec:introduction}
The Chvátal-Gomory procedure was the first cutting-plane procedure
introduced (in \cite{GY1,GY2,GY3,Chvatal}) and has been studied thoroughly from a
theoretical as well as a practical point of view (see e.g., \cite{CCH,CF1,BEHS,letchford2002strengthening,ES,fischetti2007optimizing,C,bonami2008projected}). 
Recall that the Chvátal-Gomory closure \(K'\) of a polyhedron or
a compact convex set \(K \subseteq \R^n\) is defined as
\[K' \coloneqq \bigcap_{\substack{(c,\delta) \in \Z^{n} \times \R, \\
K \subseteq \face{cx \leq \delta}}} \face{cx \leq
\ffloor{\delta}},\]
where we use \(\face{ax \leq b}\) as a shorthand for \(\face{x
  \mid ax \leq b}\); for brevity we refer to it as \gc closure and
to the defining inequalities as \gc cuts.
One of the fundamental
questions in cutting-plane theory is whether the closure arising from a
cutting-plane procedure (i.e., adding all potential cuts that can be
derived from valid inequalities) is a polyhedron. Clearly, we
add an infinite number of cuts here and thus it is not clear \emph{a
  priori} whether \(K'\) is a polyhedron. However, for the case where
\(K\) is a rational polyhedron it is well-known that the \gc closure is a rational
polyhedron again  (see \cite{schr1980},
\cite{Chvatal}). As a natural consequence, in \cite{schr1980} the question was raised
whether the \gc closure of an irrational
polytope \(P\) is a polytope. This important question was answered
in the affirmative independently in the works by 
\cite{ds2010} and \cite{DDV2011} (the latter established the even more
general case of arbitrarily compact convex
set). The relevance of this result is many-fold, from the convergence
of adding cutting planes of the \gc type to Mixed-Integer Nonlinear Programming over compact convex
sets to the theory of proof systems where we consider proofs of
assertions with infinitely many defining sentences.

\subsection*{Contribution}
\label{sec:contribution}
We provide a short proof of the more general result \cite{DDV2011}. In
contrast to \cite{DDV2011}, we do not rely so much on convex
analysis but take a rather direct topological approach while
maintaining the overall high-level strategy. Key is here a
strengthened, quantitative homogeneity lemma (see
Lemma~\ref{lem:homogeneity}) from which many required properties
follow immediately. Before, these had been established separately using
different strategies. We believe that the proposed proof
lends itself to potential applications to
many more classes of cutting planes.

The proof consists of three steps:

\begin{enumerate}

\item \emph{Continuity:} 
 Faces and
  implied cuts deform smoothly when
  perturbing the coefficients. This is crucial for the
  actual finiteness argument via compactness, and is valid for
  inequalities in general. See Lemma~\ref{lem:nearby-dir}.

\item \emph{Homogeneity} for a procedure \(M\):
 The cutting-plane closure \(M\) commutes with intersections with faces, i.e.,
 \[M(K \cap F) = M(K) \cap F,\]
 where \(F\) is a face-defining hyperplane of the convex body
 \(K\). See Lemma~\ref{lem:homogeneity}.

 Not only does homogeneity hold for many
 cutting-plane procedures \(M\) (such as the split closure, the
 Lovász-Schrijver closures, Sherali-Adams hierarchies, and the
 Lift-and-Project closure, see \cite{PS20091}) in the case of rational
 polyhedra, but
 homogeneity also allows a very clean, inductive approach
to polyhedrality. See Theorem~\ref{lem:boundary}.

Moreover, it is also homogeneity that ensures
that finitely many \gc cuts suffice to restrict the \gc closure
to a rational subspace of the affine space spanned by the convex set
(which is necessary for
polyhedrality here). See Corollary~\ref{cor:bringIntoRationalSpace}.
\item \emph{Locality:}  Informally,
  every point \(x\) in the relative interior of a
  \emph{polytope} \(P\) can be cut out by  a finite number of
  \gc cuts.

  This step is contained in the proof
  of Theorem~\ref{thm:cg-closure-polytope}.
\end{enumerate}

From all those properties the hardest one to establish and the
cornerstone of our proof is the quantitative version of the
\hyperref[lem:homogeneity]{homogeneity lemma}
in Section~\ref{sec:homogeneity}. 
In fact, for the proof we need a generalization
of a famous theorem due to Kronecker and Weyl provided in
Lemma~\ref{lem:vectorSpace}. Once homogeneity is established, the
conclusion of polyhedrality follows naturally in our
framework and it is actually very similar to the proof for rational
polyhedra given in \cite{schr1980}.

\section{Preliminaries}
\label{sec:preliminaries}
In the following, we only consider exposed faces of convex sets
and for the sake of
brevity we refer to them as \emph{faces}. 
In other words, a face \(F\) of a compact convex set \(K\)
is a subset of the form \(F = K \cap \face{\pi x = \pi_0}\)
for some supporting hyperplane \(\pi x = \pi_0\),
i.e., \(K \subseteq \face{\pi x \leq
\pi_0}\), and there exists \(x_0 \in K\) with \(\pi x_0 = \pi_0\);
we will call the face \(F\) the \emph{\(\pi\)-face of \(K\)}.
In particular, \(F = K\) is allowed if \(K\) is lower
dimensional. Recall that a compact convex set is uniquely
determined by its exposed proper faces
(i.e., exposed faces different from \(K\)).
All facts that we mention without pointers to the literature can be
found in \cite{schrijver1998theory} and \cite{barvinok2002course}.

To formalize continuity of directions,
we identify a direction with the unit vector pointing to that
direction,
therefore
for any non-zero vector \(\pi\),
let \(\unitvector{\pi} \coloneqq \pi \mathbin{/} \tnorm{\pi}\).
The next lemma shows that compact \(\pi\)-faces
change \textquote{upper semi-continuously} in the direction of
\(\pi\). In the following we denote the dimension of the ambient space
\(\R^n\) by \(n\) and we use the shorthand \([k] \coloneqq \face{1, \dots,k}\).
Let \(\partial K\) denote the relative boundary of
a closed convex set \(K\).

\begin{lem}[Continuity]
  \label{lem:nearby-dir}
  Let \(K\) be a closed convex set in \(\mathbb{R}^{n}\).
  Let \(\pi\) be a non-zero vector,
  and let the \(\pi\)-face \(F\) of \(K\) be compact.
  Then for every neighbourhood \(U\) of \(F\)
  there exists an \(\varepsilon > 0\) such that
  whenever
  \(\tnorm{\unitvector{\pi'} - \unitvector{\pi}}
  < \varepsilon\),
  the \(\pi'\)-face of \(K\) is a subset of \(U\),
  i.e.,
  all the maximizers of
  the function \(x \mapsto \pi' x\) on \(K\) lie in \(U\).
\begin{proof}
Without loss of generality,
we may assume that \(U\) is compact and convex.
By assumption, for some constant \(c\),
we have \(\pi x = c\) for all \(x \in F\),
and \(\pi x < c\) for all \(x \in K \setminus F\),
in particular for all \(x \in \partial U \cap K\).
As \(\partial U \cap K\) is compact, actually
\(\pi x < c_{1}\) for some \(c_{1} < c\)
and all \(x \in \partial U \cap K\).

As \(\pi' x\) is continuous in \(\pi'\) and \(x\),
and \(F\) and \(\partial U \cap K\)
are compact,
for all \(\pi'\) in a neighbourhood of \(\pi\),
we have \(\pi' x > c_{1}\) for all \(x \in F\),
but  \(\pi' x < c_{1}\) for all \(x \in \partial U \cap K\).
In particular,
all maximizers of \(x \mapsto \pi' x\) on \(K\) lie in \(U\).
This is obvious if \(K \subseteq U\).
If \(K \nsubseteq U\),
then \(\pi' x\) is everywhere smaller on \(K \setminus U\)
than on \(F\):
take arbitrary points \(x_{0} \in K \setminus U\)
and \(x_{1} \in F\).
There is an \(x_{2} \in \partial U \cap K\)
in the line segment \([x_{0}, x_{1}]\).
As \(\pi' x_{2} < \pi' x_{1}\),
we obtain \(\pi' x_{0} < \pi' x_{1}\).
This finishes the proof of the lemma.
\end{proof}
\end{lem}

\begin{rem}
  If \(K\) is a polyhedron one can show more:
  the \(\pi'\)-face is contained in \(F\).
  In
  particular, there is no need for \(U\). To see this we choose \(U\)
  to be a polytope in the proof.
Then \(U \cap K\) is also a polytope,
and \(x \mapsto \pi' x\) is everywhere larger
on the vertices
of the \(\pi\)-face than on the other vertices,
when the direction of \(\pi'\) is close to that of \(\pi\).
Hence the \(\pi'\)-face is contained in the \(\pi\)-face,
as claimed.

\end{rem}

We will use a well-known approximation theorem due to Kronecker.
We state a version suitable for our needs,
which we derive from Weyl’s criterion.

\begin{lem}[{\citep{kronecker1884},\citep[Satz~3]{weyl1916}}]
\label{lem:vectorSpace}
  Let \(n, N_{0} \in \N\) and \(\pi \in \R^{n}\) with \(\pi \neq 0\).
  Then
  \(\Z^{n} + \pi \Z_{> N_{0}}\) contains a dense subset of
  a linear subspace \(V\) of \(\R^{n}\).
  In particular,  \(\Z^{n} + \pi \Z_{> N_{0}}\) contains points
  arbitrarily close to \(0\),
  i.e., for every \(\epsilon > 0\)
  there exists \(N > N_{0}\) and \(a \in \Z^{n}\)
  with \(\tnorm{a - N \pi} < \epsilon\).
\begin{proof}
When the components of \(\pi\) together with \(1\)
are linearly independent over \(\Q\),
this is a special case of Weyl’s criterion with \(V = \R^{n}\).
We reduce the general case to this one.

First we define \(V\).
Let \(\pi_{1}, \dots, \pi_{n}\) denote the components of \(\pi\).
We can assume without loss of generality that
a linear basis of \(1, \pi_{1}, \dots, \pi_{n}\) over \(\mathbb{Q}\)
is \(1, \pi_{1}, \pi_{2}, \dots, \pi_{k}\).

Thus for \(j > k\) there are integers \(n_{j,i}\) and \(n_{j}\)
together with a positive integer \(m\)
such that
\begin{align*}
  m
  \pi_{j} &= n_{j} + \sum_{i=1}^{k} n_{j,i} \pi_{i}, & j &> k.
  \intertext{We use these as
    the defining equations of \(V\),
    i.e., \(V\) is defined by}
  m
  x_{j} &= \sum_{i=1}^{k} n_{j,i} x_{i}, & j &> k.
  \intertext{Let \(e_{1}, \dots, e_{n}\) denote
    the canonical basis of \(\Z^{n}\).
    The following elements lie in \(V\):}
  \widetilde{e_{i}} &\coloneqq m e_{i} + \sum_{j=k+1}^{n} n_{j,i} e_{j},
  & i &< k, \\
  \widetilde{\pi} &\coloneqq m \pi - \sum_{j=k+1}^{n} n_{j} e_{j}.
\end{align*}
By Weyl’s criterion,
\(\Z^{k} + (\pi_{1}, \dots, \pi_{k}) \Z_{> N_{0}}\)
is dense in \({\mathbb{R}}^{k}\),
and hence also
\(m \Z^{k} + m (\pi_{1}, \dots, \pi_{k}) \Z_{> N_{0}}\)
is dense in \({\mathbb{R}}^{k}\).
We reformulate this for \(V\)
via the projection to the first \(k\) coordinates,
which is obviously an isomorphism
between \(V\) and \({\mathbb{R}}^{k}\):
a dense subset of \(V\) is
\(\sum_{i=1}^{k} \Z \widetilde{e_{i}}
+ \widetilde{\pi} \Z_{> N_{0}}\),
which is a subset of \(\Z^{n} + \pi \Z_{> N_{0}}\).
This finishes the proof.
(For \(k = 0\) the argument above is overkill,
as \(\pi\) is rational and hence
\(\Z^{n} + \pi \Z_{> N_{0}}\) contains \(0\).)
\end{proof}
\end{lem}

\section{Homogeneity}
\label{sec:homogeneity}
In this section we compare \(K'\) with the \gc closure \(F'\)
of a face \(F\).

\begin{lem}[Homogeneity for compact faces]
  \label{lem:homogeneity}
  Let \(K \subseteq \R^n\) be a closed convex set and let
  \[
  F \coloneqq K \cap \face{ \pi x = \pi_0}
  \]
  be a compact \(\pi\)-face of \(K\)
  for some \(\pi \in \R^{n}\) and \(\pi_0 \in \R\)
  with \(K \subseteq \face{\pi x \leq \pi_0 }\). 
  Further, assume that \(F\) satisfies \(c x \leq \delta\) with \(c \in
  \Z^{n}\) (and hence \(F'\) satisfies \(cx \leq \ffloor{\delta}\)).
  Then there are finitely many \gc cuts of \(K\)
  defining a polyhedron \(P\) satisfying
  \((c + \alpha \pi) x \leq \ffloor{\delta} + \alpha \pi_{0}\)
  for some \(\alpha > 0\).
\begin{proof}
Rescaling \((\pi_0,\pi)\) we may assume, without loss of generality, that \(\pi_{0} \in \Z\).
By increasing \(\delta\) a little if necessary,
we may assume that \(\delta\) is not an integer.
Note that a small increase of \(\delta\)
does not change \(\ffloor{\delta}\).
Let \(0 < \varepsilon < \delta - \ffloor{\delta}\) be
a small positive number.
Choose a small compact neighbourhood \(U\) of \(F\)
such that \(c x \leq \delta + \varepsilon\) for \(x \in U\).

There is clearly an \(\varepsilon_{1} > 0\)
such that for all \(y \in \R^{n}\)
with \(\tnorm{y} < \varepsilon_{1}\),
we have
\(\abs{y x} < \varepsilon\) for all \(x \in U\).
By Lemma~\ref{lem:nearby-dir},
there exists an \(\varepsilon_{2} > 0\) such that
whenever
\(\tnorm{\unitvector{\pi'} - \unitvector{\pi}}
< \varepsilon_{2}\),
all maximizers of the function \(x \mapsto \pi' x\) on \(K\)
lie in \(U\).
There is a large positive integer \(N\)
such that for all positive integer \(m \geq N\)
and \(a \in \Z^{n}\)
with \(\tnorm{a - m \pi} < \varepsilon_{1}\)
we have
\(\tnorm{\unitvector{c+a} - \unitvector{\pi}} = \tnorm{\unitvector{\frac{c+a}{m}} - \unitvector{\pi}}
< \varepsilon_{2}\).
In particular,
all maximizers of \((c + a) x\) on \(K\) lie in \(U\)
for all such \(m\) and \(a\).
All in all,
for all positive integer \(m \geq N\)
and \(a \in \Z^{n}\)
with \(\tnorm{a - m \pi} < \varepsilon_{1}\),
all maximizers of the function \(x \mapsto (c + a) x\)
on \(K\) lie in \(U\),
and we have \(\tnorm{(a - m \pi) x} < \varepsilon\)
for all \(x \in U\).

Now we choose a finite collection of such pairs \((m, a)\).
By Lemma~\ref{lem:vectorSpace}
the collection \(\Z^{n} - \Z_{{} \geq N} \pi\)
contains a dense subset of
a linear subspace \(V\) of \(\R^{n}\).
Let \(v_{1}, \dotsc, v_{k}\) be the vertices of a small simplex in \(V\)
containing \(0\) in its relative interior,
with \(\tnorm{v_{i}} < \varepsilon_{1}\) for all \(i\).
We choose the simplex to be full dimensional in \(V\),
i.e., here \(k-1\) is the dimension of \(V\).
By slightly perturbing the \(v_{i}\) in \(V\),
the vertices remain in the \(\varepsilon_{1}\)-ball,
and \(0\) in the interior of the simplex.
As \(\Z^{n} - \Z_{{} \geq N} \pi\) contain a dense subset of \(V\),
by a slight perturbation we can even move the vertices
inside \(\Z^{n} - \Z_{{} \geq N} \pi\),
obtaining a new simplex with vertices \(a_{i} - m_{i} \pi\),
with \(0\) still contained
in the relative interior of the new simplex:
i.e., there exist coefficients \(\lambda_{i}\) satisfying
\begin{equation}
  \label{eq:5}
  \sum_{i \in [k]} \lambda_{i} (a_{i} - m_{i} \pi) = 0,
\qquad \text{where }\lambda_{1},\dots,\lambda_k > 0 \text{ and } \sum_{i \in [k]}\lambda_i = 1
\end{equation}
with some \(a_{i} \in \Z^{n}, \ m_i \in \N,\ m_{i} \geq N\)
satisfying \(
  \tnorm{a_{i} - m_{i} \pi} < \varepsilon_{1}\).

As a consequence, for all \(x \in U \cap K\), one has
\begin{equation}
  \label{eq:2}
  (c + a_{i}) x = c x + m_{i} \pi x + (a_{i} - m_{i} \pi) x
  \leq (\delta + \varepsilon) + m_{i} \pi_{0} + \varepsilon,
\end{equation}
which is also valid for all \(x \in K\) as \((c + a_{i})x\)
attains its maximum in \(U\).
Hence
\begin{equation*}
  (c + a_{i}) x \leq \ffloor{\delta + m_{i} \pi_{0} + 2 \varepsilon}
  = \ffloor{\delta} + m_{i} \pi_{0}
\end{equation*}
is a \gc cut for \(K\) for \(i \in [k]\).

Let \(P\) be the polyhedron defined by these \gc cuts.
The convex combination of the \gc cuts
with coefficients \(\lambda_{i}\)
is valid for \(P\),
which is exactly the claimed inequality 
\begin{equation*}
  (c + \alpha \pi) x \leq \ffloor{\delta} + \alpha \pi_{0},
\end{equation*}
for \(P\) with \(\alpha \coloneqq \sum_{i \in [k]} \lambda_{i} m_{i} > 0\) as
\eqref{eq:5} can be rewritten to
\begin{equation*}
  \sum_{i \in [k]} \lambda_{i} (c + a_{i}) =
  c + \underbrace{\sum_{i \in [k]} \lambda_{i}  m_{i}}_{\alpha} \pi.
  \qedhere
\end{equation*}
\end{proof}
\end{lem}

Apart from establishing the basis for the later induction, the main
advantage of Lemma~\ref{lem:homogeneity} is that many
important properties of the \gc closure follow as simple corollaries.

\begin{cor}
  \label{cor:face}
  Let \(K \subseteq \R^n\) be a compact convex set.
  Then \(K' \subseteq K\) and we have \(K' \cap F = F'\) for every face \(F\) of \(K\).
\begin{proof}
Applying Lemma~\ref{lem:homogeneity} to \(c = 0\) and \(\delta =0\),
we obtain that \(K'\) satisfies
every inequality \(\pi x \leq \pi_{0}\) satisfied by \(K\).

For a face \(F\), Lemma~\ref{lem:homogeneity} implies that \(K' \cap F\)
satisfies the \gc cuts defining \(F'\),
hence \(K' \cap F \subseteq F'\).
The inclusion in the other direction \(F' \subseteq K' \cap F\)
is obvious.
\end{proof}
\end{cor}

\begin{rem}
If one merely
wants to establish \(K' \cap F \subseteq (K \cap F)'\) and one is not interested in
the finiteness statement of Lemma~\ref{lem:homogeneity}
then it suffices to consider a single vector \(c + a_1\)
in the proof of Lemma~\ref{lem:homogeneity}
instead of a finite family.
From \eqref{eq:2} the proof can then be concluded as
follows: For every \(x \in K' \cap F\) 
\begin{equation*}
  c x = (c + a_1) x + (m_1 \pi - a_1) x - m_1 \pi_{0}
  \leq \ffloor{\delta} + \varepsilon
\end{equation*}
for every \(\varepsilon > 0\) small enough.
Thus \(c x \leq \ffloor{\delta}\) is valid for \(K' \cap F\).

Also note that \(K' \subseteq K\) alternatively follows with \cite[Lemma
2]{DeyP2010} (see \cite{Dadush:de:vi:10} for a similar result).
\end{rem}

We further obtain

\begin{cor}
  \label{cor:bringIntoRationalSpace}
  Let \(K\) be a compact convex set.
  Then finitely many \gc cuts of \(K\) define
  a polyhedron in a rational affine subspace \(V\)
  with \(V \subseteq \aff{K}\).
\begin{proof}
The affine subspace \(\aff{K}\) is defined by
finitely many inequalities \(a_i x \leq  b_i\)
with \(i \in [\ell]\) for some
\(\ell \in \N\).
These are consequences of finitely many \gc cuts
via Lemma~\ref{lem:homogeneity} with \(\pi = a_i\), \(\pi_0 = b_i\)
and \(c = 0\), \(\delta = 0\).
Therefore the polyhedron defined by these \gc cuts spans
a rational affine subspace \(V\) of \(\aff{K}\).
\end{proof}
\end{cor}

\section{The \gc closure of a compact convex set}
\label{sec:gc-closure-compact}

We will now prove the main theorem:
\begin{thm}
  \label{thm:cg-closure-polytope}
  Let \(K\) be a compact convex set.
  Then \(K'\) is a rational polytope
  defined by finitely many \gc cuts of \(K\).
\end{thm}

The proof will proceed via induction on the dimension of \(K\) using
the following step lemma. 

\begin{lem}
  \label{lem:boundary}
  Let \(K\) be a compact convex set.
  Let us assume that for every proper face \(F\) of \(K\),
  the \gc closure \(F'\) is defined by finitely many \gc cuts of \(F\)
  (i.e., Theorem~\ref{thm:cg-closure-polytope} holds for \(F\)).
  Then
  there is a polytope \(P\) defined by finitely many \gc cuts of \(K\),
  which is contained in \(K\)
  and \(P \cap \partial K = K' \cap \partial K\).
\begin{proof}
For all unit vectors \(\pi\) in the lineality space of \(\aff{K}\),
let \(\pi x \leq \pi_{0}\) define
the associated supporting hyperplane of \(K\).
Now \(F_{\pi} = K \cap \face{ \pi x = \pi_0}\)
is a proper face, the \(\pi\)-face of \(K\),
and hence \(F_{\pi}'\) is defined
by finitely many \gc cuts of \(F_{\pi}\) by our assumption.
By Lemma~\ref{lem:homogeneity} there are finitely many \gc cuts
of \(K\) defining a polyhedron \(P_{\pi}\)
satisfying \(\pi x \leq \pi_{0}\)
and \(P_{\pi} \cap \face{ \pi x = \pi_0} = F_{\pi}'\).
For \(F_{\pi}' \neq \emptyset\) this means exactly
that the \(\pi\)-face of \(P_{\pi}\) is \(F_{\pi}'\).
By adding finitely many \gc cuts,
we may assume that \(P_{\pi}\) is a polytope.

We claim that
for vectors \(\pi'\) in a neighbourhood \(U_{\pi}\) of \(\pi\),
the polytope \(P_{\pi}\) still satisfies
\(\pi' x \leq \pi'_{0}\)
and \(P_{\pi} \cap \face{ \pi' x = \pi'_0} = F_{\pi'}'\).
This is immediate
if \(F_{\pi}' = \emptyset\),
as \(F_{\pi'}\) is disjoint from \(P_{\pi}\)
by Lemma~\ref{lem:nearby-dir}.
If \(F_{\pi}' \neq \emptyset\),
let \(F_{P_{\pi}, \pi'}\) denote the \(\pi'\)-face of \(P_{\pi}\).
The inequality \(\pi' x \leq \pi'_{0}\)
is satisfied by \(P_{\pi}\) as \(F_{P_{\pi}, \pi'}\)
is also the \(\pi'\)-face of \(F'_{\pi}\).
If \(F_{P_{\pi}, \pi'}\) is not contained
in the \(\pi'\)-face \(F_{\pi'}\) of \(K\),
then \(P_{\pi}\) satisfies \(\pi' x < \pi'_{0}\) and
\(F_{\pi'}' = \emptyset\).
However,
if \(F_{P_{\pi}, \pi'}\) is contained in the \(\pi'\)-face \(F_{\pi'}\),
then clearly
\(F_{\pi'}' = F_{P_{\pi}, \pi'} = P_{\pi} \cap \face{ \pi' x = \pi'_{0}}\),
proving the claim.

We obtain an open cover of the unit
sphere of \(\aff{K}\) with neighborhoods \(U_\pi\)
such that for each \(\pi' \in U_\pi\)
we have
\(P_{\pi} \cap \face{ \pi' x = \pi'_0} = F_{\pi'}'\)
and \(\pi' x \leq \pi'_{0}\) for all \(x \in P_{\pi}\).
Since the unit sphere is compact,
it follows by choosing a finite subcover that
finitely many \gc cuts define a polytope \(Q\)
with
\(Q \cap \face{ \pi x = \pi_0} = F_{\pi}'\)
and \(\pi x \leq \pi_{0}\) for all \(x \in Q\).

By Corollary~\ref{cor:bringIntoRationalSpace},
by adding finitely many cuts
we obtain a polytope \(P\) in a rational affine subspace of \(\aff{K}\),
which is contained in \(Q\).
In particular, it lies in \(K\) and
\(P \cap \partial K = K' \cap \partial K\).
\end{proof}
\end{lem}

Finally, we are ready to prove the main theorem.
\begin{proof}[Proof of Theorem~\ref{thm:cg-closure-polytope}]
The proof proceeds via induction on the dimension of \(K\).
By the induction hypothesis, the Theorem holds
for proper faces of \(K\).
From Lemma~\ref{lem:boundary} we know that
finitely many \gc cuts define a polytope \(P\) in \(K\)
with \(P \cap \partial K = K' \cap \partial K\).
The polytope \(P\) spans
a rational affine subspace \(V\) of \(\aff{K}\).

Let \(D\) denote the orthogonal projection
of \(\Z^{n}\) onto the lineality space \(W\) of \(V\).
As \(W\) is rational, the orthogonal projection \(D\)
is a lattice.
We claim that there are only finitely many \(d \in D\)
with a preimage \(c \in \Z^{n}\),
for which a \gc cut \(c x \leq \ffloor{\delta}\)
cuts out something from \(P\)
i.e., at least one vertex \(v\) of \(P\).

As vertices on the boundary of \(K\) belong to \(K'\),
these cannot be cut out. Therefore \(v\) has to be contained in the relative interior of
\(K\) and so does a small ball \(U\) in \(V\) around \(v\).
Let \(r\) denote the radius of \(U\).
Now whenever \(d \in D\) is too long, i.e., \(\tnorm{d} \geq 1/r\),
we have \(\max_{x \in K} c x \geq \max_{x \in U} cx \geq cv + 1\) as
\(x-v \in W\) and so \(cx-cv = dx - dv\). Hence \(cx \leq \ffloor{\max_{x
  \in K} cx}\) cannot cut off \(v\). As there are only a finite number
of vertices \(v\) of \(P\),
there is a global upper bound on the length of the \(d\)
which could cut out a vertex in the relative interior of \(K\).
As \(D\) is discrete, there are only finitely many such vectors \(d\),
and adding these \gc cuts to \(P\) we obtain \(K'\).

Actually, for every \(d\) we need to add only one cut,
thus defining \(K'\) by finitely many \gc cuts, as claimed.
To prove this,
we consider all the \gc cuts \(c x \leq \ffloor{\delta}\)
where \(c\) is a preimage of a fixed \(d \in D\).
We claim that unless \(P = \emptyset\) 
(when \(P = K'\) and the theorem holds),
there is a deepest cut among these.
This will be the only cut
we need to add to \(P\) for the vector \(d\).

To prove the last claim, let \(x_{0}\) be a rational point of \(V\).
Restricted to \(V\), every \gc cut \(c x \leq \ffloor{\delta}\)
can be rewritten to \(d x \leq \ffloor{\delta} - (c - d) x_{0}\).
As \(\ffloor{\delta}\) is an integer,
\(c - d\) is rational with bounded denominator
(as \(D\) is discrete),
and \(x_{0}\) is rational,
therefore
the right-hand side can take only a discrete set of values,
and \(d x\) is a lower bound on the set of values for every
\(x \in P\).
Therefore there is a cut with \(\ffloor{\delta} - (c - d) x_{0}\)
minimal, which is obviously a deepest cut.
\end{proof}

\section*{Acknowledgements}
\label{sec:acknowledgements}

The authors are grateful to Daniel Dadush and Santanu Dey for
helpful comments and discussions and to the anonymous reviewers whose
comments improved the presentation significantly.

\bibliographystyle{abbrvnat}

\bibliography{bibliography}

\end{document}